\documentclass[12pt]{amsart}

\usepackage{latexsym}
\usepackage{amsmath,amssymb}
\usepackage{amscd}

\newtheorem{thm}{Theorem}[section]

\newtheorem{cor}[thm]{Corollary}
\newtheorem{prop}[thm]{Proposition}
\newtheorem{lem}[thm]{Lemma}
\newtheorem{ques}[thm]{Question}

\theoremstyle{definition}

\theoremstyle{remark}

\newcommand{\cl}[1]{\operatorname{Cl}_{#1}}
\newcommand{\naibu}[1]{\operatorname{Int}_{#1}}

\newcommand{\fr}[1]{\operatorname{Fr}_{#1}}

\newcommand{\diam}{\operatorname{diam}}

\newcommand{\sdim}{\operatorname{dim}^{\infty}}
\newcommand{\sind}{\operatorname{ind}}
\newcommand{\lind}{\operatorname{Ind}}

\newcommand{\pois}{\hspace{-2pt}\smallsetminus\hspace{-2pt}}

\newcommand{\mapright}[1]{%
\smash{\mathop{%
\hbox to 1cm{\rightarrowfill}}\limits^{#1}}}
\newcommand{\mapleft}[1]{%
\smash{\mathop{%
\hbox to 1cm{\leftarrowfill}}\limits^{#1}}}

\usepackage[dvips]{graphicx,color}

\begin{document}

\title[Dimension of the Smirnov remainder]
{Large inductive dimension of 
the Smirnov remainder}

\author[Y. Akaike et al]
{Yuji Akaike, Naotsugu Chinen and Kazuo Tomoyasu}

\address{Kure National College of Technology,
2-2-11 Aga-Minami Kure-shi Hiroshima 737- 
\indent 8506, Japan}
 
\email{akaike@kure-nct.ac.jp}

\address{Okinawa National College of Technology,
905 Henoko Nago-shi Okinawa 905-2192, 
\indent Japan}
 
\email{chinen@okinawa-ct.ac.jp}

\address{Miyakonojo National College of Technology,
473-1 Yoshio-cho Miyakonojo-shi 
\indent Miyazaki 885-8567, Japan}

\email{tomoyasu@cc.miyakonojo-nct.ac.jp}


\thanks{An earlier version of this paper was presented at 
S\=urikaisekikenky\=usho in Kyoto University, October 11, 2005.}

\keywords{Large inductive dimension; Smirnov compactification}

\subjclass[2000]{Primary 54D35, 54D40; Secondary 54F45}

\begin{abstract}
The purpose of this paper is to investigate 
the large inductive dimension 
of the remainder of the Smirnov
compactification of $\mathbb R^n$ with the usual metric,
and give an application of it.
\end{abstract}
 
\maketitle

\section{Introduction}

\medskip

We follow the notation and terminology 
of \cite{engelking} and \cite{engelking2}.
We say that two compactifications 
$\alpha X$ and $\gamma X$ of a space $X$ 
are {\it equivalent} provided that
there exists a homeomorphism $f:\alpha X\to\gamma X$ such that
$f|_X$ is the identity map on $X$, and we denote this by
writing $\alpha X\approx\gamma X$. 
As usual, $X\cong Y$ means that $X$ is homeomorphic to $Y$.
Let $Y$ be a subspace of a metric space $(X,d)$.
We denote by $d|_Y$ the subspace metric on $Y$ induced by $d$.
A metric $d$ on $X$ is said to be {\it proper}
if for every $r>0$, $\cl{X}B_r(x,d)$ is compact,
where $B_r(x,d)=\{y\in X:d(x,y)<r\}$.
We use $\mathbb R$, $\mathbb Q$, $\mathbb Z$ and $\mathbb N$ 
for the reals, the rationals, the integers, and the natural numbers.
We write $\mathbb J$ and $\mathbb I$
for $[0,\infty)$ and $[0,1]$. \par

By $C^{\ast }(X)$, we denote the Banach algebra of
all bounded real-valued continuous functions on 
a space $X$ with the sup-norm.
It is well-known that there is a one-to-one correspondence between
the compactifications of a space $X$ 
and the closed subrings of $C^{\ast }(X)$
containing the constants and generating the topology of $X$.
Let $U^{\ast }_d (X)$ be the set of all bounded uniformly 
continuous functions of a metric space $(X,d)$.
Then we note that $U^{\ast }_d (X)$ is a closed subring of $C^{\ast }(X)$.
The {\it Smirnov compactification} $u_d X$ of a  metric space $(X,d)$
is the unique compactification associated with the closed subring 
$U^{\ast }_d(X)$ of $C^{\ast }(X)$ and is metric-dependent. 
Now, we recall the construction of the Smirnov compactifications:
Let $e:X\to \prod_{f\in U_d^{\ast}(X)}I_f$ be the evaluation map
associated with $U_d^{\ast}(X)$,
where $I_f=[{\rm inf}f(X), {\rm sup}f(X)]\subset \Bbb R$.
Recall that, identifying $X$ with $e(X)$,
the closure $\cl{}(e(X))$ in $\prod_{f\in U_d^{\ast}(X)}I_f$
and the Smirnov compactification $u_d X$ are equivalent.
See \cite{kn} and \cite{woods} for more details.
Here, recall the following fact concerning 
the Smirnov compactifications.

\medskip

\begin{prop}\cite[Theorem 2.5]{woods}\label{prop:separation}
Let $X$ be a noncompact metric space with a metric $d$.
Then the following conditions are equivalent$:$
\begin{enumerate}
\item[$(1)$] 
A compactification $\alpha X$ of $X$ is equivalent to $u_d X$,
\item[$(2)$] 
For disjoint closed sets $A, B\subset X$, 
$\cl{\alpha X}A\cap\cl{\alpha X}B=\emptyset$ 
if and only if $d(A,B)>0$.
\end{enumerate}
\end{prop}

\medskip

The aim of this paper is to investigate 
the large inductive dimension of 
the remainder of the Smirnov compactification. 
The following result is well-known as Smirnov's theorem
(See \cite{smirnov} or \cite[p.256]{kn}):
$\dim u_d X\pois X=\sdim (X,d)$ holds
for each noncompact metric space $(X,d)$,
where $\sdim (X,d)$ is the boundary dimension of $(X,d)$.
Notice that Smirnov's theorem does not explain
the large inductive dimension of $u_dX\pois X$.
As usual, if $d$ is proper, then $u_dX\pois X$ contains
a copy of $\mathbb N^*$ that is 
the Stone-{\v C}ech remainder of $\mathbb N$.
Thus, we don't know whether $\dim u_dX\pois X=\lind u_dX\pois X$
or $\dim u_dX\pois X<\lind u_dX\pois X$. 
In section 2, we calculate the large inductive dimension 
$\lind u_{d_n}\mathbb R^n \pois \mathbb R^n$
of the remainder of the Smirnov compactification of $\mathbb R^n$
with the usual metric $d_n$, and
we show that $\lind u_{d_n} \mathbb R^n \pois \mathbb R^n = n$ 
for each $n\in\mathbb N$.
Furthermore, for any noncompact
locally compact connected separable metrizable space $X$,
we show that for each $n\in\mathbb N$,
there exists a compatible proper
metric $d_n$ on $X$ such that 
$\dim u_{d_n}X\pois X=\sind u_{d_n}X\pois X=
\lind u_{d_n}X\pois X=n$.
In section 3, for any noncompact
locally compact separable metrizable space $X$, 
we show that there exists a totally bounded 
metric $d_n$ on $X$ such that 
$\dim u_{d_n}X\pois X=\sind u_{d_n}X\pois X=
\lind u_{d_n}X\pois X=n$ for each $n\in\mathbb Z$ with $n\ge 0$, 
and give an approximation to the Stone-{\v C}ech compactification.

\bigskip

\section{The Smirnov remainder generated by a proper metric}

\bigskip

Let $X$ be a topological space and $A,B$ a pair of
disjoint subsets of $X$. We say that a closed
set $L\subset X$
is a {\it partition} 
between $A$ and $B$ if there exist open sets
$U,~V\subset X$ such that 
$A\subset U,~B\subset V,~U\cap V=\emptyset$, and 
$X\pois L=U\cup V$.
Recall that a normal space $X$ satisfies the inequality
$\lind X\le n(\ge0)$ if and only if for every pair
$A,B$ of disjoint closed subsets of $X$ there exists a 
partition $L$ between $A$ and $B$ such that $\lind L\le n-1$.\par

Let $(X,d)$ and $(Y,\rho)$ be metric spaces.
A bijection $f:X\to Y$ is called a 
{\it uniform isomorphism} if both $f$ and $f^{-1}$ are
uniformly continuous. In this case we say that 
the metric spaces $(X,d)$ and $(Y,\rho)$ are
{\it uniformly equivalent}. \par

\bigskip

\begin{lem}\label{lem:subspace}
Let $(X,d)$ and $(Y,\rho)$ be noncompact metric spaces and
$f : Y \rightarrow X$ a uniform closed embedding 
$($i.e., $f : Y \rightarrow f(Y)$ is a uniform isomorphism 
and $f(Y)$ is closed$)$.
Then the following statements hold$:$
\begin{enumerate}
\item[(1)] $u_{\rho}Y \pois Y$ is embedded in $u_{d}X \pois X$.
\item[(2)] Let $(Y_1,d_1)$ be a compact metric space and
$(Y_2,d_2)$ a noncompact metric space.
If $Y$ is uniformly equivalent to $(Y_1 \times Y_2,d_1+d_2)$, 
then $Y_1$ is embedded in $u_{d}X \pois X$. 
\end{enumerate}
\end{lem}

\begin{proof}
(1) It follows immediately by 
\cite[Theorem 2.9 and 2.10]{woods}.\par

(2) It follows from (1) that 
$u_{\rho}Y \pois Y$ is embedded in $u_{d}X \pois X$.
Since $Y_1$ is a compact metric space 
and $Y$ is uniformly equivalent to $Y_1 \times Y_2$, 
by \cite[Theorem 2.10 and 3.6]{woods},
$u_{\rho} Y\cong Y_1 \times u_{d_2}Y_2$.
Thus, $u_{\rho} Y \pois Y\cong Y_1 \times (u_{d_2}Y_2 \pois Y_2)$,
and then $Y_1$ is embedded in $u_{d}X \pois X$.
\end{proof}

\bigskip

Let $n,m\in\mathbb N$ and $k \in \mathbb Z$ with $n \geq k \geq 0$.
Denote subspaces $Z_{k}^n$ and $Z_{k}^{n,m}$ 
of $\mathbb R^n$ as follows:
$$Z_{k}^n=\{(x_1,x_2,\dots,x_n) \in \mathbb R^n : 
| \{i : x_i \in \mathbb Z \}| \geq n-k \}\mbox{ and}$$
$$Z_{k}^{n,m}=\{(x_1,x_2,\dots,x_n) \in \mathbb R^n : 
| \{i : mx_i \in \mathbb Z \}| \geq n-k \}\mbox{.}$$

\bigskip

\begin{lem} For each $n,m\in\mathbb N$ and 
each $k\in\mathbb Z$
with $n \geq k \geq 0$, we have the following facts$:$
\begin{enumerate}
\item[(1)] $Z_k^n$ is $k$-dimensional.
\item[(2)] $Z_k^n$ and $Z_k^{n,m}$ are uniformly equivalent.
\end{enumerate}
\end{lem}

\begin{proof}
(1) In fact, since $Z_0^n=\mathbb Z^n$ and $Z_n^n=\mathbb R^n$,
the cases that either $k=0$ or $k=n$ is true. 
Suppose that neither $k=0$ nor $k=n$. 
Here, we denote by $[X]^{\kappa}$ the set of all subsets of $X$ with
the cardinality $\kappa$. Observe that
$$Z_k^n=\bigcup_{I\in [\{1,\ldots,n\}]^{n-k}}
\prod_{i\in I}{\mathbb Z}_i\times
\prod_{i\not\in I}{\mathbb R}_i,$$
where ${\mathbb Z}_i=\mathbb Z$ and ${\mathbb R}_i=\mathbb R$
for $i=1,2,\ldots,n$. This indicates that
$Z_k^n$ is the union of countable subspaces 
each of which are homeomorphic to $\mathbb R^k$.
Then by
\cite[Theorem 3.1.8]{engelking2} $\dim Z_k^n\le k$. 
On the other hand, since $\mathbb R^k$ is closed embedded in $Z_k^n$, 
$\dim Z_k^n=k$.\par
(2) Define $f:Z_k^{n,m}\to Z_k^n$ by
$f(x)=mx$ for each $x\in Z_k^{n,m}$.
Clearly, since $f$ is a uniform isomorphism, 
$Z_k^n$ and $Z_k^{n,m}$ are uniformly equivalent.
\end{proof}

We are now ready to present our main result.

\bigskip

\begin{thm}\label{thm:Ind}
Suppose that $n \in \mathbb N,~k \in \mathbb Z$
and $n \geq k \geq 0$. 
If $d$ is the metric on $Z_{k}^n$ induced
by the usual metric on $\mathbb R^n$,
then $\lind  u_d Z_{k}^n \pois Z_{k}^n = k$.
\end{thm}

\begin{proof}
We may assume that
$d((x_1,x_2,\ldots,x_n),(y_1,y_2,\ldots,y_n)) 
= \sum_{i = 1}^n|x_i - y_i|$ (cf. \cite[p.47]{woods}).
Since $Z_{0}^n$ is 1-discrete for each $n \in \mathbb N$,
$u_dZ^n_0\approx\beta Z_{0}^n$ which 
is the Stone-{\v C}ech compactification of $Z_{0}^n$
(cf. \cite[Theorem 3.4]{woods}). 
Then $\lind  u_d Z_{0}^n \pois Z_{0}^n = 0$
(cf. \cite[Theorem 7.1.11 and 7.1.17]{engelking}).
Fix an $n \in \mathbb N$ with $n \geq k$.
Suppose that $\lind  u_d Z_{i}^n \pois Z_{i}^n = i$ 
for each $i$ with  $0 \leq i \leq k-1$.
We only need to show that $\lind u_dZ_k^n\pois Z_k^n=k$.

Let 
$$Y_j = [0,1]^{k-1}\times [2j,2j+1]
\times \{(\underbrace{0,0,\dots,0}_{n-k})\}
\subset Z_{k}^n~\mbox{and}~Y = \bigcup_{j \geq 1}Y_j \subset Z_{k}^n.$$
Since $[0,1]^k\times \mathbb N$ and $Y$ are uniformly equivalent
with suitable metrics, 
by Lemma \ref{lem:subspace}, 
we see that $[0,1]^k$ is embedded in $u_d Z_{k}^n \pois Z_{k}^n$,
thus,
$\lind  u_d Z_{k}^n \pois Z_{k}^n \geq k$ by 
\cite[Theorem 2.2.1]{engelking2}.

Let $A$ and $B$ be disjoint closed subsets in $u_d Z_{k}^n \pois Z_{k}^n$.
We show that there exists a partition $L'$  
between $A$ and $B$ with $\lind L' \leq k-1$.\par

Since $u_d Z_{k}^n$ is normal,
there exist open subsets $U$ and $V$ of $u_d Z_{k}^n$ such that
$A \subset U$, $B \subset V$, and 
$\cl{u_d Z_{k}^n}U \cap \cl{u_d Z_{k}^n}V = \emptyset$.
By Proposition \ref{prop:separation},
$\varepsilon = d(Z_{k}^n \cap \cl{u_d Z_{k}^n}U, 
Z_{k}^n \cap \cl{u_d Z_{k}^n}V) > 0$. 
Choose an $m \in \mathbb N$ with $5n/m < \varepsilon$. 
Let $\Lambda  = \{\cl{Z_{k}^n}C : C$ 
is a component of $Z_{k}^n \pois Z_{k-1}^{n,m}\},$
$W'_0 = \bigcup\{D \in \Lambda : 
D \cap \cl{u_d Z_{k}^n}U \neq \emptyset \}$,
$W_0 = \bigcup\{D \in \Lambda : 
D \cap W_0' \neq \emptyset \}$,
$W_1 =  \cl{Z_{k}^n} (Z_{k}^n \pois W_0)$, and
$L = W_0 \cap W_1$.
Note that $L={\rm Fr}_{Z_{k}^n} W_0 
=  {\rm Fr}_{Z_{k}^n} W_1$
is contained in $Z_{k-1}^{n,m}$.

{\bf Fact.} $\cl{u_d Z_{k}^n} L 
= \cl{u_d Z_{k}^n} W_0 \cap \cl{u_d Z_{k}^n} W_1$. \par

Note that $\cl{u_d Z_{k}^n} L 
\subset \cl{u_d Z_{k}^n} W_0 \cap \cl{u_d Z_{k}^n} W_1$ always holds. 
Suppose that we have a point 
$x \in \cl{u_d Z_{k}^n} W_0 \cap \cl{u_d Z_{k}^n} W_1 
\pois \cl{u_d Z_{k}^n} L$. 
Since $Z_{k}^n \cap \cl{u_d Z_{k}^n} L 
= Z_{k}^n \cap \cl{u_d Z_{k}^n} W_0 \cap \cl{u_d Z_{k}^n} W_1$,
$x \in u_d Z_{k}^n \pois Z_{k}^n$. 
Since $u_d Z_{k}^n$ is normal,
there exists a closed neighborhood
$S \subset u_d Z_{k}^n$ such that $x \in S$ and 
$S \cap \cl{u_d Z_{k}^n} L = \emptyset$. 
By Proposition \ref{prop:separation}, 
$d(W_0 \cap S, W_1 \cap S) = 0$
because $\cl{u_d Z_{k}^n}(W_0\cap S)\cap\cl{u_d Z_{k}^n}(W_1\cap S)
\not=\emptyset$.
Thus, there exist sequences 
$x_{i,0},x_{i,1},\ldots \in W_i \cap S$ 
for $i = 0,1$ such that
$\lim_{k \rightarrow \infty} d(x_{0,k},x_{1,k}) = 0$.
We may assume that $x_{i,0},x_{i,1},\ldots \not\in L$ for $i = 0,1$.
Then there exists an arc $P_k$ joining $x_{0,k}$ and $x_{1,k}$
in  $Z_{k}^n$ with $\lim_{k \rightarrow \infty} \diam P_k = 0$.
Since $L =  {\rm Fr}_{Z_{k}^n} W_0 =  {\rm Fr}_{Z_{k}^n} W_1$,
we can take an element $y_k \in L \cap P_k$ for each $k \in \mathbb N$.
Since $\lim_{k \rightarrow \infty} d(x_{0,k},y_{k}) 
= \lim_{k \rightarrow \infty} d(x_{1,k},y_{k}) = 0$,
$d(W_0 \cap S, L) =d(W_1 \cap S, L) = 0$.
By Proposition \ref{prop:separation},
we have $\cl{u_d Z_{k}^n}(W_i \cap S)
\cap \cl{u_d Z_{k}^n} L \neq \emptyset$ 
for $i=0,1$. This contradicts the fact that 
$S \cap \cl{u_d Z_{k}^n} L = \emptyset$,
as claimed.

Then, it suffices to show that
$L' =  \cl{u_d Z_{k}^n}L \pois L$ is a partition between $A$ and $B$.
Let $W^*_0 = u_d Z_{k}^n \pois \cl{u_d Z_{k}^n}W_1$,
$W^*_1 = u_d Z_{k}^n \pois \cl{u_d Z_{k}^n}W_0$ and
$X_i = W^*_i \pois Z_{k}^n$ for $i = 0,1$.
Notice that
\begin{align*}
u_d Z_{k}^n \pois \cl{u_d Z_{k}^n} L 
& =  u_d Z_{k}^n \pois (\cl{u_d Z_{k}^n}W_0 
\cap \cl{u_d Z_{k}^n} W_1)\\
& =  (u_d Z_{k}^n \pois \cl{u_d Z_{k}^n}W_0) 
\cup (u_d Z_{k}^n \pois \cl{u_d Z_{k}^n} W_1)\\
& =  W^*_0 \cup W^*_1.
\end{align*}

\noindent
This shows that 
$(u_d Z_{k}^n \pois Z_{k}^n) \pois L' = X_0 \cup X_1$,
and then we only need to show that 
$A \subset X_0$ and $B \subset X_1$. 
Notice that
$d(Z_{k}^n \cap \cl{u_d Z_{k}^n}U, W_1) \geq 1/m$.
By Proposition \ref{prop:separation},
$\cl{u_d Z_{k}^n}U \cap \cl{u_d Z_{k}^n}W_1 = \emptyset$,
and then $A \subset X_0$.
We show that 
$d(W_0,Z_{k}^n \cap \cl{u_d Z_{k}^n}V) \geq 2n/m$.
Suppose that 
$d(W_0,Z_{k}^n \cap \cl{u_d Z_{k}^n}V) < 2n/m$.
Since 
$W_0 \subset B_{3n/m}(Z_{k}^n \cap \cl{u_d Z_{k}^n}U,d)$,
$d(Z_{k}^n \cap \cl{u_d Z_{k}^n}U,
Z_{k}^n \cap \cl{u_d Z_{k}^n}V) < 5n/m$.
By the definition of $\varepsilon$, 
this is a contradiction.
So, by Proposition \ref{prop:separation},
$\cl{u_d Z_{k}^n}V \cap \cl{u_d Z_{k}^n}W_0 = \emptyset$,
and then $B \subset X_1$.

Now, since $L \subset Z_{k-1}^{n,m}$,
by Lemma \ref{lem:subspace},
$L'$ is embedded in $u_d Z_{k-1}^{n,m} \pois Z_{k-1}^{n,m}$.
Since $Z_{k-1}^{n,m}$ and $Z_{k-1}^n$ are uniformly equivalent, 
by \cite[Theorem 2.2.1]{engelking2},
$\lind L'\leq \lind u_d Z_{k-1}^{n,m}\pois Z_{k-1}^{n,m}
 = \lind u_d Z_{k-1}^n\pois Z_{k-1}^n = k-1$. 
By \cite[Proposition 1.6.2]{engelking2},
$\lind  u_d Z_{k}^n \pois Z_{k}^n \leq k$ 
for each $n \in \mathbb N$ and each $k\in\mathbb Z$ with $n\ge k\ge0$.
\end{proof}

\medskip

\begin{cor}\label{cor:largedim}
Let $n\in\mathbb N$ and $d$ the usual metric on 
$X=\mathbb R^n$ or $\mathbb J^n$.
Then $\sind u_d X \pois X   
= \lind  u_d X \pois X 
= \dim u_d X \pois X  = n$.
\end{cor}

\begin{proof}
By Lemma \ref{lem:subspace} (1),
$u_dX\pois X$ is embedded in $u_d Z_{n}^n \pois Z_{n}^n$.
Thus, by \cite[Theorem 1.6.3, 2.2.1 and 3.1.28]{engelking2} 
and Theorem \ref{thm:Ind},
$$\max\{\sind u_d X\pois X, \dim u_d X\pois X \} 
\leq  \lind u_d X\pois X
\leq  \lind u_dZ_{n}^n \pois Z_{n}^n = n.$$
Let $Y_k = \mathbb I^{n-1} \times[2k,2k + 1]$ 
and $Y = \bigcup_{k \geq 1}Y_k$.  
Since $Y$ and $\mathbb I^n \times \mathbb N$ 
are uniformly equivalent with suitable metrics, 
by Lemma \ref{lem:subspace} (2),
$\mathbb I^n$ is embedded in $u_d X\pois X$.
By \cite[Theorem 1.1.2 and 3.1.3]{engelking2},
$$n \leq \min \{\sind u_d X\pois X, \dim u_d X\pois X\}
\leq  \lind  u_d X\pois X.$$
Thus, $\sind u_d X\pois X  = \lind  u_d X\pois X= \dim u_d X\pois X = n$.
\end{proof}

\medskip

As an application of Corollary \ref{cor:largedim}
we show that for any $n\in\mathbb N$
and any noncompact,
locally compact, connected, separable metrizable space $X$ 
there exists a proper metric $d_n$ on $X$ such that 
$\dim u_{d_n}X\pois X=\sind u_{d_n}X\pois X=\lind u_{d_n}X\pois X=n$.

\medskip

\begin{lem}\label{lem:approx}
Let $(X,d)$ and $(Y,\rho)$ be 
noncompact, connected,
proper metric spaces.
If there exists a perfect map $p:\mathbb J\to Y$ 
such that $u_{\rho|_{p(\mathbb J)}}p(\mathbb J)\pois p(\mathbb J)
\cong u_\rho Y\pois Y$, 
then there exists a proper metric 
$d_X$ compatible with the topology on $X$
such that $u_{d_X}X\pois X\cong u_\rho Y\pois Y$.
\end{lem}

\begin{proof}
Fix an $x_0\in X$. Define $f:X\to\mathbb J$
by $f(x)=d(x,x_0)$ for each $x\in X$. Since $d$ is proper,
$f$ is a perfect onto map. Then $g=p\circ f:X\to Y$ is a perfect map.
Let $\omega X$ be the one-point compactification of $X$
and let $\omega X=X\cup\{p_{\infty}\}$ as a set.
Define $G:X\to\omega X\times Y$ by
$G(x)=(x,g(x))$ for each $x\in X$. \par

{\bf Fact 1}. $G:X\to\omega X\times Y$ is a closed embedding.\par
It suffices to show that $G$ is closed.
Let $A$ be a closed subset of $X$.  
Suppose the contrary that
$\cl{\omega X\times Y}G(A)\pois G(A)\not=\emptyset$.
Take a point $(x,y)\in\cl{\omega X\times Y}G(A)\pois G(A)$.
We may assume that $x=p_\infty$.
There exists a sequence
$\{(x_n,g(x_n))\}_{n\in\mathbb N}\subset G(A)$  such that
$x_n\to p_{\infty}$ and $g(x_n)\to y$ if $n\to\infty$.
Since $\lim_{n\to\infty}f(x_n)=\infty$, 
we may assume that $f(x_n)>\max\{n,f(x_1),\ldots,f(x_{n-1})\}$ 
for each $n\in\mathbb N$. 
By $\sigma$-compactness of $Y$, there exists a compact cover
$\{K_n\}_{n\in\mathbb N}$ of $Y$ such that
$B_1(K_n,\rho)\subset\naibu{Y}K_{n+1}$ for each $n\in\mathbb N$.
Here, there exists an $N\in\mathbb N$ such that $y\in K_N$. 
Since $g$ is a perfect map,
$g(\{x_n\}_{n\in\mathbb N})\pois K_n\not=\emptyset$ 
for each $n\in\mathbb N$, and then
$\lim_{n\to\infty}\rho(g(x_n),y)=\infty$.
This is a contradiction, as claimed.\par

Let $d_{\omega X}$ be a metric on $\omega X$, and let
$\sigma((x,y),(x',y'))=d_{\omega X}(x,x')+\rho(y,y')$ for each
$(x,y),(x',y')\in\omega X\times Y$. 
Put $d_X(x,x')=\sigma((x,g(x)),(x',g(x')))$ for each $x,x'\in X$.
Since $\sigma$ is a proper metric on $\omega X\times Y$,
by Fact 1, $d_X$ is a proper metric compatible with the topology on $X$.
Moreover, we note that $(X,d_X)$ and $(G(X),\sigma|_{G(X)})$ are uniformly equivalent.
Thus, by \cite[Theorem 2.9 and 2.10]{woods}, 
$u_{d_X}X\pois X\cong\cl{u_{\sigma}(\omega X\times Y)}G(X)\pois G(X)$.
\par

{\bf Fact 2}. 
(1) $\lim_{n\to\infty}\sup
\{\sigma(z,G(X)):
z\in\{p_{\infty}\}\times p(\mathbb J)\pois\omega X\times K_n\}=0$.
\par
\noindent(2) $\lim_{n\to\infty}\sup
\{\sigma(z,\{p_{\infty}\}\times p(\mathbb J)):
z\in G(X)\pois\omega X\times K_n\}=0$.
\par
Notice that for each $n\in\mathbb N$ 
there exists an $\ell_n\in\mathbb N$
such that $g(X\pois B_{1/n}(p_{\infty},d_{\omega X}))
\linebreak\subset K_{\ell_n}$ 
and $K_{\ell_n}\subset K_{\ell_m}$
whenever $n<m$.
Now, we show (1). Fix an $n\in\mathbb N$ and 
take a point $z=(p_{\infty},p(t))\not\in\omega X\times K_m$
with $m\ge \ell_n$. Then 
$g(X\pois B_{1/n}(p_{\infty},d_{\omega X}))
\subset K_{\ell_n}\subset K_m$.
Since $p(t)\not\in K_m$, 
$p(t)\not\in g(X\pois B_{1/n}(p_{\infty},d_{\omega X}))$,
and then $p(t)\in 
g(B_{1/n}(p_{\infty},d_{\omega X})\pois \{p_{\infty}\})$.
As $f$ is surjective, there exists an 
$x\in B_{1/n}(p_{\infty},d_{\omega X})\pois\{p_\infty\}$
such that $p(t)=p(f(x))=g(x)\in 
g(B_{1/n}(p_{\infty},d_{\omega X})\pois\{p_\infty\})$.
Since $(x,g(x))\in G(X)$,
$$\sigma(z,G(X))\le\sigma(z,(x,g(x)))=
\sigma((p_{\infty},p(t)),(x,p(t)))
=d_{\omega X}(p_{\infty},x)<1/n.$$
Then the proof of (1) is complete. 
To prove (2), take a point 
$z=(x,g(x))\in G(X)\pois \omega X\times K_m$.
We may assume that $m\ge \ell_n$. Since there exists a $t\in\mathbb J$
such that $p(t)=g(x)$,
$$\sigma(z,\{p_{\infty}\}\times p(\mathbb J))
\le \sigma(z,(p_{\infty},p(t)))
=\sigma((x,p(t)),(p_{\infty},p(t)))
=d_{\omega X}(x,p_{\infty})<1/n.$$
Then the proof of (2) is complete.\par
By \cite[Theorem\ 2.9 and 2.10]{woods}, Fact 2 and
\cite[Theorem 4.2]{woods},
\begin{align*}
u_{d_X}X\pois X 
\cong & ~~\cl{u_{\sigma}(\omega X\times Y)}G(X)\pois G(X)\\
    = & ~~\cl{u_{\sigma}(\omega X\times Y)}(\{p_{\infty}\}
          \times p(\mathbb J))
       \pois (\{p_{\infty}\}\times p(\mathbb J))\\
\cong & ~~u_{\sigma|_{\{p_{\infty}\}\times p(\mathbb J)}}
          (\{p_{\infty}\}\times p(\mathbb J))
          \pois (\{p_{\infty}\}\times p(\mathbb J))\\ 
\cong & ~~u_{\rho|_{p(\mathbb J)}}p(\mathbb J)
          \pois p(\mathbb J)\\
\cong & ~~u_{\rho}Y\pois Y.
\end{align*}

Then the proof is complete.
\end{proof}

\medskip

\begin{ques}
Let $(X,\rho)$ be a connected proper metric space
such that $u_\rho X \pois X$ is connected. 
Does there exist a compatible metric $d$ on $\mathbb J$
such that $u_{d}\mathbb J\pois \mathbb J\cong u_\rho X \pois X$?
\end{ques}

\medskip

A metrizable space is said to be a
{\it continuum} provided that
it is compact and connected. A space is said to be a
{\it Peano continuum} provided that it is a
locally connected continuum.\par

\medskip

\begin{thm}\label{thm:dimremainder}
Let $X$ be a noncompact, locally compact, connected, 
separable metrizable space.
Then for any $n\in\mathbb N$ there exists a proper metric $d_n$ 
compatible with the topology on $X$
such that $\dim u_{d_n}X\pois X=
\sind u_{d_n}X\pois X=\lind u_{d_n}X\pois X=n$. 
\end{thm}

\begin{proof}
By Corollary \ref{cor:largedim} and Lemma \ref{lem:approx}
we only need to construct a perfect map $p:\mathbb J\to \mathbb J^n$ 
such that $u_{\rho|_{p(\mathbb J)}}p(\mathbb J)\pois p(\mathbb J)
\cong u_\rho \mathbb J^n\pois \mathbb J^n$, where
$\rho$ is the usual metric on $\mathbb J^n$. \par
Let $K_i=[0,i]^n$ and $D_i=K_i\pois \naibu{\mathbb J^n}K_{i-1}$
for each $i\in\mathbb N$, where $K_0=\{\mathbf 0\}$. 
Choose $\ell_i \in \fr{\mathbb J^n}{K_i}$ for each $i \in\mathbb N$.
There exists a surjective continuous map 
$p_i : [i-1,i] \rightarrow D_i$ such that $p_i(i-1) = \ell_i$
and $p_i(i)=\ell_{i+1}$ because $D_i$ is a Peano continuum.
Define $p: \mathbb J \rightarrow \mathbb J^n$ by
$p=\bigcup_{i=1}^\infty p_i$. 
Note that $p$ is well-defined, and is a perfect map. 
By \cite[Theorem 2.9, 2.10 and 4.2]{woods},
$u_{\rho|_{p(\mathbb J)}}p(\mathbb J)\pois p(\mathbb J)\cong 
u_\rho \mathbb J^n\pois \mathbb J^n$, as claimed.
\end{proof}

\bigskip

\section{Smirnov remainder generated by a totally bounded metric}

\bigskip

A metric $d$ on $X$ is said to be {\it totally bounded}
provided that $(X,d)$ is a totally bounded metric space.
In this section, we consider totally bounded metric spaces, 
and provide a counterpart of Theorem \ref{thm:dimremainder} 
and an approximation of the Stone-{\v C}ech compactification.

\medskip

\begin{prop}
Let $X$ be a noncompact, locally compact, separable metrizable space.
Then there exists a totally bounded metric $d_n$ compatible
with the topology on $X$ such that 
$\dim u_{d_n}X\pois X=\sind u_{d_n}X\pois X=\lind u_{d_n}X\pois X=n$
for each $n\in\mathbb Z$ with $n\ge0$.
\end{prop}

\begin{proof} 
In the case that $n=0$ is clear. 
By Aarts and Van Emde Boas' result (cf. \cite{ab} and \cite{ss}),
for any $n\in\mathbb N$
there exists a compactification $\alpha X$ of $X$
such that $\alpha X\pois X\cong\mathbb I^n$.
Note that $\alpha X$ is metrizable because
$\alpha X$ has countable network weight.
Then, let $\rho$ be a suitable metric on $\alpha X$
and let $d_n=\rho|_X$. By Proposition \ref{prop:separation},
$\alpha X\approx u_{d_n} X$, as claimed.
\end{proof}

\medskip

A compact Hausdorff space is a {\it weak Peano space} if 
it contains a dense, continuous image of the real line $\mathbb R$.
Note that every weak Peano space is not necessarily metrizable.
For example, the Stone-{\v C}ech compactification $\beta\mathbb R$
is a nonmetrizable weak Peano space.\par
Here, we show the following lemma which is an
elaborate version for the
locally compact separable metrizable spaces 
concerning Theorem 3 in \cite{ch}.

\medskip

\begin{lem}\label{lem:singular}
Let $X$ be a noncompact, locally compact, separable 
metrizable space and $K$ a nondegenerate metrizable weak Peano space. 
Let $A$ and $B$ be disjoint noncompact closed subsets of $X$.
If there exists a closed copy $N$ of $\mathbb N$ in $X$
with $N\cap(A\cup B)=\emptyset$, 
then there exists a totally bounded metric $d$
compatible with the topology on $X$ such that $u_d X\pois X\cong K$
and $\cl{u_d X}A\cap\cl{u_d X}B=\emptyset$.
\end{lem}

\begin{proof}
Take points $p,q\in K$ with $p\not=q$.
Let $U_p$ and $U_q$ be neighborhoods
of $p$ and $q$ respectively
such that $\cl{K}U_p\cap\cl{K}U_q=\emptyset$.
Since $K$ is a weak Peano space, there exists a continuous map
$f:\mathbb R\to K$ such that $f(\mathbb R)$ is dense in $K$.
Take points $r_p\in f^{-1}(U_p)\cap \mathbb Q$ and
$r_q\in f^{-1}(U_q)\cap \mathbb Q$. 
Here, enumerate $\mathbb Q$ and $N$ as $\{q_n:n\in\mathbb N\}$
and $\{x_n:n\in\mathbb N\}$, respectively.
Define $\varphi:N\cup A\cup B\to \mathbb Q$ by
\[
\varphi(x)=
\begin{cases}
r_p,                 
\quad\quad &\text{if}\quad x\in A,\\
r_q,  
\quad &\text{if}\quad x\in B,\\
q_n,
\quad &\text{if}\quad x=x_n.
\end{cases}
\]
\noindent
Then there exists a continuous extension 
$\psi:X\to \mathbb R$
such that $\psi|_{N\cup A\cup B}=\varphi$. 
Put $g=f\circ\psi$.
Define $e:X\to \omega X\times K$
by $e(x)=(x,g(x))$
for each $x\in X$.
Clearly, $e$ is an embedding, 
and $\alpha X=\cl{\omega X\times K}e(X)$ is a compactification
of $X$ with $K$ as a remainder. 
Let $d_{\omega X}$ and $d_K$ be metrics compatible with the topology on
$\omega X$ and $K$, respectively.
Here, define a metric $s$ on $\omega X\times K$
by $s((x,u),(y,v))=d_{\omega X}(x,y)+d_K(u,v)$ for each 
$(x,u),~(y,v)\in\omega X\times K$. 
Now, let $d$ be a metric on $X$ induced by $s$, i.e.,
$d(x,y)=s((x,g(x)),(y,g(y)))$ for each $x,y\in X$. 
Clearly, $d$ is totally bounded.
By Proposition \ref{prop:separation}, $\alpha X\approx u_d X$. 
Furthermore, since $\cl{\alpha X}e(A)\cap\cl{\alpha X}e(B)\subset
(\omega X\times\cl{K}U_p)\cap(\omega X\times\cl{K}U_q)=\emptyset$,
$\cl{u_d X}A\cap\cl{u_d X}B=\emptyset$
by \cite[Theorem 3.5.5]{engelking}.
\end{proof}

\medskip

Let $\operatorname{TBM}(X)$ be the set of all compatible
totally bounded metrics on $X$. 
Here, we have the following approximation 
to the Stone-{\v C}ech compactification.

\medskip

\begin{prop}
Let $X$ be a noncompact, locally compact, separable metrizable space.
Then we have the following approximation to
the Stone-{\v C}ech compactification of $X$
for each $n\in \mathbb N$$:$
$$\beta X\approx\sup\{u_{\rho} X : \dim u_{\rho} X\pois X
=\sind u_{\rho} X\pois X=\lind u_{\rho} X\pois X=n
~\mbox{and}~\rho\in\operatorname{TBM}(X)\}.$$
\end{prop}

\begin{proof} 
Since $X$ is locally compact separable metrizable,
there exists a proper metric $d$ compatible with the topology
on $X$ (cf. \cite[Lemma 3.1]{kt}). Let $A$ and $B$ be disjoint closed subsets of $X$.
We may assume that neither $A$ nor $B$ is compact.
Then we can take a closed copy $N=\{x_n:n\in\mathbb N\}$ of $\mathbb N$
in $A$ such that $\{B_{1/2}(x_n,d):n\in\mathbb N\}$ 
is discrete in X because $A$ is not compact. 
Put $A_0=A\pois\bigcup_{n\in\mathbb N}B_{1/2}(x_{2n-1},d)$
and $A_1=A\pois\bigcup_{n\in\mathbb N}B_{1/2}(x_{2n},d)$.
Since $\{A_0,B\}$ and $\{A_1,B\}$ are two pairs of disjoint
closed subsets of $X$ and both $N_0=\{x_{2n}:n\in\mathbb N\}$ and 
$N_1=\{x_{2n-1}:n\in\mathbb N\}$ 
are closed copies of $\mathbb N$,
we can apply Lemma \ref{lem:singular} to
$\{A_i,B\}$, $N_j$ and $\mathbb I^n$ for $i,j=0,1$ with $i\not=j$. 
Fix an $n\in\mathbb N$.
Then there exists a totally bounded metric $d_i$ on $X$
compatible with the topology on $X$ such that
$u_{d_i}X\pois X\cong \mathbb I^n$ and 
$\cl{u_{d_i}X}A_i\cap\cl{u_{d_i}X}B=\emptyset$ for $i=0,1$.
Then put $\gamma X=\sup\{u_{d_0}X,~u_{d_1}X\}$. 
Since $\cl{u_{d_i}X}A_i\cap\cl{u_{d_i}X}B=\emptyset$ 
for $i=0,1$, $\cl{\gamma X}A\cap\cl{\gamma X}B=\emptyset$. Define
$$\delta X=\sup\{u_{\rho} X : 
u_{\rho} X\pois X\cong \mathbb I^n~\mbox{and}
~\rho\in\operatorname{TBM}(X)\}.$$
\noindent
Since $\delta X\ge \gamma X$, 
$\cl{\delta X}A\cap\cl{\delta X}B=\emptyset$ for
each disjoint closed subsets $A,B$ of $X$. 
This is a characterization of the Stone-{\v C}ech 
compactification of the normal spaces, and then 
$\delta X\approx\beta X$, as claimed.
\end{proof}


{\small \begin{thebibliography}{99}



\bibitem{ab}J.M. Aarts and P. Van Emde Boas,
{\it Continua as remainders in compact extension}, 
Nieuw Archief voor Wiskunde
15 $(1967)$, 34-37.



\bibitem{ch}R.E. Chandler,
{\it Continua as remainders, revisited},
Gen. Top. Appl. 8 $(1978)$, 63-66.



\bibitem{engelking}R. Engelking,
{\it General Topology},
Helderman Verlag, Berlin, 1989.



\bibitem{engelking2}R. Engelking,
{\it Theory of Dimension Finite and Infinite},
Helderman Verlag, Berlin, 1995.



\bibitem{kt}K. Kawamura and K. Tomoyasu,
{\it Approximations of Stone-{\v C}ech compactifications by Higson compactifications}, 
Coll. Math.  88 $(2001)$, 75-92.



\bibitem{kn}Y. Kodama and K. Nagami,
{\it General topology $($Japanese$)$}, Iwanami, Tokyo, 1974.



\bibitem{smirnov}Ju.M. Smirnov,
{\it On the dimension of remainders in bicompact extensions of proximity and
topological spaces $($Russian$)$},
Math. Sb. (N.S.) 69 (111) $(1966)$, 141-160.



\bibitem{ss}A.K. Steiner and E.F. Steiner,
{\it Compactifications as closure of graphs}, 
Fund. Math. 63 $(1968)$, 221-223.



\bibitem{woods}R.G. Woods,
{\it The minimum uniform compactification of a metric space}, 
Fund. Math. 147 $(1995)$, 39-59.



\end{thebibliography}}
\end{document}